\documentclass[11pt]{amsart}

\usepackage{amssymb,amsmath,accents}
\usepackage{bbm}
\usepackage{graphicx}
\usepackage{a4wide}

\newtheorem{theorem}{Theorem}
\newtheorem{prop}{Proposition}
\newtheorem{lemma}{Lemma}     
\newtheorem{coro}{Corollary}
\theoremstyle{definition}

\newcommand{\ts}{\hspace{0.5pt}}
\newcommand{\nts}{\hspace{-0.5pt}}

\newcommand{\NN}{\mathbb{N}}
\newcommand{\PP}{\mathbb{P}}

\newcommand{\cA}{\mathcal{A}}
\newcommand{\cB}{\mathcal{B}}
\newcommand{\cC}{\mathcal{C}}

\newcommand{\cM}{\mathcal{M}}
\newcommand{\cP}{\mathcal{P}}

\newcommand{\pa}{\hphantom{g}\nts\nts}
\newcommand{\vph}{\vphantom{I}}

\newcommand{\dd}{\,\mathrm{d}}
\newcommand{\ee}{\ts\mathrm{e}}
\newcommand{\one}{\mathbbm{1}}

\newcommand{\pmin}{\ts\ts\underline{\nts\nts 0\nts\nts}\ts\ts}
\newcommand{\pmax}{\ts\ts\underline{\nts\nts 1\nts\nts}\ts\ts}
\newcommand{\bigtim}{\mbox{\LARGE $\times$}}

\begin{document}

\title[Solving the recombination equation via Haldane linearisation]
{Haldane linearisation done right:\ Solving the \\[2mm]
nonlinear  recombination equation the easy way}

\author{Ellen Baake}
\address{Technische Fakult\"at, Universit\"at Bielefeld, 
         Postfach 100131, 33501 Bielefeld, Germany}

\author{Michael Baake}
\address{Fakult\"at f\"ur Mathematik, Universit\"at Bielefeld, 
         Postfach 100131, 33501 Bielefeld, Germany}

\begin{abstract} 
  The nonlinear recombination equation from population genetics has a
  long history and is notoriously difficult to solve, both in
  continuous and in discrete time. This is particularly so if one aims
  at full generality, thus also including degenerate parameter
  cases. Due to recent progress for the continuous time case via the
  identification of an underlying stochastic fragmentation process, it
  became clear that a direct general solution at the level of the
  corresponding ODE itself should also be possible.  This paper shows
  how to do it, and how to extend the approach to the discrete-time
  case as well.
\end{abstract}

\maketitle

\section{Introduction}

The \emph{recombination equation} is a well-known dynamical system
from mathematical population genetics \cite{Lyu,Chris,Buerger,EB-ICM},
which describes the evolution of the genetic composition of a
population that evolves under recombination. The genetic composition
is described via a probability distribution (or measure) on a space of
sequences of finite length, and recombination is the genetic mechanism
in which two parent individuals are involved in creating the mixed
sequence of their offspring during sexual reproduction.  The model
comes in a continuous-time and a discrete-time version. It can
accommodate a variety of different mechanisms by which the genetic
material of the offspring is partitioned across its parents.  In all
cases, the resulting equations are nonlinear and notoriously difficult
to solve. Elucidating the underlying structure and finding solutions
has been a challenge to theoretical population geneticists for nearly
a century now.

The first studies go back to Jennings in 1917 \cite{Jennings} and
Robbins in 1918 \cite{Robbins}.  Geiringer in 1944 \cite{Gei} and
Bennett in 1954 \cite{Ben} were the first to state the generic general
\emph{form} of the solution in terms of a convex combination of
certain basis functions, and developed methods for the recursive
evaluation of the corresponding coefficients to obtain the solution
itself, at least in principle.  The approach was later continued
within the systematic framework of genetic algebras; compare
\cite{Lyu, HaleRingwood}. It could be shown that, despite the
nonlinearity, the dynamical system may be (exactly) transformed into a
linear one by embedding it into a higher-dimensional space. More
explicitly, a large number of further components are added that
correspond to multilinear transformations of the original
measure. This method is known as \emph{Haldane linearisation}
\cite{HaleRingwood}.  However, this line of research led to
astonishingly few concrete or applicable results. This raises the
question whether the program may be completed outside the abstract
framework, and what kinds of results can be obtained via different
approaches.

A first step forward, for an important special case, was achieved in
\cite{BB}, via a rather powerful use of the inclusion-exclusion
principle in the form of the M\"{o}bius inversion formula. At the same
time, a general formalism via nonlinear operators, called
\emph{recombinators}, was introduced that also allowed for an
alternative consideration starting from the nonlinear equation with a
single such operator and extending this to the same solution
\cite{MB}.

After various intermediate steps of gradual generalisations, the
complete equation, in the setting of general partitions, was analysed
and solved in \cite{BBS}. In the generic parameter case, the general
solution was given in recursive form. Also, the principal form of the
solution in degenerate cases was analysed, but no general formula was
given. The most important insight, however, was the identification of
an underlying stochastic fragmentation process in \cite[Sec.~6]{BBS}.
This means that the solution of the nonlinear recombination equation
has a representation in terms of the solution of the Kolmogorov
forward equation for this very process, which is a \emph{linear} ODE.

In \cite{interval}, the slightly simpler setting of ordered or
interval partitions was analysed, with focus on explicit solution
formulas for \emph{all} parameter values (thus including the
degenerate cases) for sequences of length up to five. Here, the
above-mentioned Kolmogorov equation was investigated further, and the
Markov generator from \cite[Sec.~6]{BBS} was derived explicitly. This
demonstrated two important things:
\begin{enumerate}
\item The solvability of the nonlinear recombination ODE ultimately
  rests upon the fact that this solution essentially also solves a
  system of linear equations;\vspace{2mm}
\item The degenerate cases are in one-to-one correspondence to the
  cases where the Markov generator fails to be diagonalisable, and the
  appearance of Jordan blocks, well known from classic ODE theory,
  determines the solutions then.
\end{enumerate}
Now, with hindsight, one can ask whether one can treat the original
nonlinear ODE in such a way that this becomes immediately transparent,
without resorting to the underlying stochastic process. The answer is
affirmative, and this paper explains how to do it. Effectively, this
new approach means to re-interpret the original Haldane linearisation
in a suitable way, without any need for genetic algebras.  Moreover,
as we shall see, a completely analogous approach also works for the
discrete-time recombination equation.  \smallskip

This paper builds on previous work, most importantly on
\cite{BBS,interval}. Some of the results from these papers will be
freely used below, and not re-derived here (though we will always
provide precise references). Also, the biological background is
explained in \cite{BBS}.  After recalling the preliminaries and our
notation in Section~\ref{sec:partitions}, the general recombination
equation in continuous time, together with its reduction to subsystems
via marginalisation, is discussed in Section~\ref{sec:gen-reco}. This
is followed by its general solution via our new and simplified
strategy (Section~\ref{sec:fast}), which leads to the first main
result in Theorem~\ref{thm:main}. A stratified interpretation in terms
of the underlying partitioning process is offered in
Section~\ref{sec:partitioning}, which gives our second main result
(Theorem~\ref{thm:a-part}).

The discrete-time version of the recombination equation is then
discussed and solved in Section~\ref{sec:discrete}, by the same method,
which leads to our third main result in Theorem~\ref{thm:discrete}. The 
corresponding stochastic process is also identified and briefly
summarised.

\section{Partitions, product spaces, measures and 
recombinators}\label{sec:partitions}

Let $S$ be a finite set, and consider the lattice  $\PP (S)$ of
partitions of $S$; see \cite{Aigner} for general background on lattice
theory and \cite{BBS} for details of the present setting.  Here, we
write a partition of $S$ as $\cA = \{ A_{1}, \dots , A_{m} \}$, where
$m = |\cA|$ is the number of its (non-empty) parts (also called
blocks), and one has $A_{i} \cap A_{j} = \varnothing$ for all $i\ne j$
together with $A_{1} \cup \dots \cup A_{m} = S$. The natural ordering
relation is denoted by $\preccurlyeq$, where $\cA \preccurlyeq \cB$
means that $\cA$ is \emph{finer} than $\cB$, or that $\cB$ is
\emph{coarser} than $\cA$.  The conditions $\cA \preccurlyeq \cB$ and
$\cB \succcurlyeq \cA$ are synonymous, while $\cA \prec\cB$ means $\cA
\preccurlyeq \cB$ together with $\cA \ne \cB$, so $\cA$ is strictly
finer than $\cB$.

The joint refinement of two partitions $\cA$ and $\cB$ is written as
$\cA \wedge \cB$, and is the coarsest partition below $\cA$ and $\cB$.
The unique \emph{minimal} partition within the lattice $\PP (S)$ is
denoted as $\pmin = \big\{ \{x\} \mid x \in S \big\}$, while the
unique \emph{maximal} one is $\pmax = \{ S \}$.  When $U$ and $V$ are
disjoint (finite) sets, two partitions $\cA\in\ts\PP(U)$ and
$\cB\in\ts\PP(V)$ can be joined to form an element of $\PP(U\nts \cup
V)$. We denote such a \emph{joining} by $\cA\sqcup \cB$, and similarly
for multiple joinings.  Conversely, if $U\nts\subseteq S$, a partition
$\cA\in\ts\PP(S)$, with $\cA = \{ A_{1}, \dots , A_{m} \}$ say,
defines a unique partition of $U$ by restriction. The latter is
denoted by $\cA|^{\pa}_{U}$, and its parts are precisely all non-empty
sets of the form $A_{i} \cap U$ with $1\leqslant i \leqslant
m$. \smallskip

Fix now $S=\{ 1,2,\dots ,n\}$ and define $X := X_{1} \times \dots
\times X_{n}$, where each $X_{i}$ is a locally compact space (which we
mean to include the Hausdorff property).  The natural projection of
$X$ to its $i$th component is denoted by $\pi_{i}$, so $\pi_{i} (X)
= X_{i}$. For an arbitrary non-empty subset $U\nts\subseteq S$, we use
the notation $\pi^{\pa}_{U} \!  : \; X \longrightarrow X^{\pa}_{U}:=
\bigtim_{i\in U} X_{i}$ for the projection to $X^{\pa}_{U}$.

Let $\cM (X)$ denote the space of signed, finite and regular Borel
measures on $X$, equipped with the usual total variation norm $\|
. \|$, which makes it into a Banach space. Also, we need the closed
subset (or cone) $\cM_{+} (X)$ of positive measures, which includes
the zero measure. Within $\cM_{+} (X)$, we denote the closed subset of
probability measures by $\cP (X)$. Note that $\cM_{+} (X)$ and
$\cP(X)$ are convex sets.  The restriction of a measure $\mu\in\cM
(X)$ to a subspace $X^{\pa}_{U}$ is written as $\pi^{\pa}_{U} . \mu :=
\mu \circ \pi^{-1}_{U}$, which is consistent with marginalisation of
measures. When the context is clear, we will use the abbreviation
$\mu^{U}\nts := \pi^{\pa}_{U}.\mu$. For any Borel set $A\subseteq
X^{\pa}_{U}$, one thus has the relation $\mu^{U}\! (A) =
\bigl(\pi^{\pa}_{U} . \mu\bigr) (A) = \mu \bigl( \pi^{-1}_{U}
(A)\bigr)$.

Given a measure $\mu \in \cM(X)$ and a partition $\cA = \{ A_{1},
\dots ,A_{m} \} \in \PP (S)$, we define the mapping $R^{\pa}_{\nts \cA} \!
: \, \cM(X)\longrightarrow \cM(X)$ by $\mu \mapsto R^{\pa}_{\nts \cA}
(\mu)$ with $R^{\pa}_{\nts \cA} (0) := 0$ and, for $\mu\ne 0$,
\begin{equation}\label{eq:def-recomb}
    R^{\pa}_{\nts \cA} (\mu) \, := \, \frac{1}{\| \mu \|^{m-1}}
    \bigotimes_{i=1}^{m} \bigl(\pi^{\pa}_{\! A_{i}} . \ts \mu\bigr)
    \, = \, \frac{\ts \mu^{A^{}_{1}} \otimes \dots \otimes
            \mu^{A^{}_{m}}}{ \| \mu \|^{m-1}}  \ts .
\end{equation}
Note that the product is (implicitly) `site ordered', which means that
it matches the ordering of the sites as specified by the set $S$. We
shall use (implicit) site ordering also for product sets. We call a
mapping of type $R^{\pa}_{\!\cA}$ a \emph{recombinator}. Note that
recombinators are nonlinear whenever $\cA \ne \pmax$.

Let us recall some results from \cite[Prop.~1 and Cor.~1]{BBS} as
follows.

\begin{prop}\label{prop:gen-props}
  Let\/ $S=\{ 1,2,\dots ,n \}$ and\/ $X=X_{1} \times \dots \times
  X_{n}$ as above. Now, let\/ $\cA \in \PP(S)$ be arbitrary, and
  consider the corresponding recombinator\/ $R^{\pa}_{\nts \cA}$ as
  defined by Eq.~\eqref{eq:def-recomb}. Then, the following assertions
  are true.
\begin{enumerate}\itemsep=1pt
\item $R^{\pa}_{\! \cA}$ is positive homogeneous of degree\/ $1$,
  which means that\/ $R^{\pa}_{\! \cA} (a \mu) = a R^{\pa}_{\! \cA}
  (\mu)$ holds for all\/ $\mu\in\cM(X)$ and all\/ $a\geqslant 0$.
\item $R^{\pa}_{\! \cA}\vph$ is globally Lipschitz on\/ $\cM(X)$, with
  Lipschitz constant\/ $L\leqslant 2\ts \lvert \cA \rvert + 1$.
\item On\/ $\cP (X)$, the recombinator\/ $R^{\pa}_{\! \cA}$ is
  Lipschitz with $L\leqslant \lvert \cA \rvert$.
\item $\| R^{\pa}_{\! \cA} (\mu)\| \leqslant \| \mu \|\vph$ holds for
  all\/ $\mu\in\cM(X)$.
\item $R^{\pa}_{\! \cA}\vph$ maps\/ $\cM_{+} (X)$ into itself.
\item $R^{\pa}_{\! \cA}\vph$ preserves the norm of positive measures,
  and hence also maps\/ $\cP (X)$ into itself.
\item On\/ $\cM_{+} (X)$, the recombinators satisfy $R^{\pa}_{\nts
    \cA} R^{\pa}_{\cB} = R^{\pa}_{\nts \cA \wedge \cB} $.  In
  particular, each recombinator is an idempotent and any two
  recombinators commute.  \qed
\end{enumerate}
\end{prop}

Several of these properties will be used below without further
mentioning, some in results that we simply recall from previous work.
Let us mention (without proof) that the Lipschitz constant in claim 2
can be improved to $L \leqslant 2\ts \lvert \cA \rvert - 1$.
We are now set to define and analyse the recombination ODE.

\section{The general recombination equation and 
marginalisation}\label{sec:gen-reco}

The general recombination equation in continuous time is formulated
within the Banach space $(\cM(X), \|.\|)$, as the nonlinear ODE
\begin{equation}\label{eq:reco-eq}
   \dot{\omega}^{\pa}_{t} \, = \!
   \sum_{\cA\in \PP (S)} \! \! \varrho(\cA) \ts
   \bigl( R^{\pa}_{\nts\cA} - \one \bigr) (\omega^{\pa}_{t})
\end{equation}
with non-negative numbers $\varrho(\cA)$ that have the meaning of
\emph{recombination rates} in our context.  We will usually assume
that an initial condition $\omega^{\pa}_{0} \in \cM(X)$ for $t=0$ is
given for the ODE \eqref{eq:reco-eq}, and then speak of the
corresponding \emph{Cauchy problem} (or initial value problem).

With $\varPhi := \sum_{\cA\in \PP (S)} \varrho(\cA) \ts \bigl(
R^{\pa}_{\nts\cA} - \one \bigr)$, we can now simply write
\begin{equation}\label{eq:short}
     \dot{\omega}^{\pa}_{t} \, = \, \varPhi (\omega^{\pa}_{t}) \ts ,
\end{equation}
but we must keep in mind that $\varPhi$ is a nonlinear operator.
Nevertheless, one has the following basic result \cite{BB,BBS}; see
\cite{Amann} for general background on ODEs on Banach spaces.

\begin{prop}\label{prop:gensol}
  Let\/ $S$ be a finite set and\/ $X$ the corresponding locally
  compact product space as introduced above. Then, the Cauchy problem
  of Eq.~\eqref{eq:reco-eq} with initial condition\/ $\omega^{\pa}_{0}
  \in \cM (X)$ has a unique solution.  Moreover, the cone\/ $\cM_{+}
  (X)$ is forward invariant, and the flow is norm-preserving on\/
  $\cM_{+} (X)$.  In particular, $ \cP (X)$ is forward invariant under
  the flow. \qed
\end{prop}

Without loss of generality, when we start with a positive measure, we
may thus restrict our attention to the investigation of the
recombination equation on the cone $\cM_{+} (X)$, and on $\cP (X)$ in
particular. So, let us assume that we consider the Cauchy problem with
$\omega^{\pa}_{0} \in \cP (X)$.

The way to a solution of the recombination equation in previous papers
started with the ansatz
\begin{equation}\label{eq:ansatz}
     \omega^{\pa}_{t} \, = \sum_{\cA\in\PP (S)} \! a^{\pa}_{t} (\cA)
      \, R^{}_{\nts \cA} (\omega^{\pa}_{0})
\end{equation}
which effectively means a complete separation of the time evolution
and the recombination of the initial condition. This led to a
nonlinear ODE system for the coefficient functions $a^{\pa}_{t} (\cA)$
that could be solved recursively in \cite{BBS}, for generic
recombination rates.  Via the identification of an underlying Markov
partitioning process in \cite[Sec.~6]{BBS}, the further analysis of
\cite{interval} showed that these coefficient functions also solve a
\emph{linear} ODE system with constant coefficient matrix, $Q$ say,
which makes the entire solvability understandable in retrospect. As
mentioned in the Introduction, the degenerate cases then correspond to
$Q$ not being diagonalisable.

In this approach, which meant a significant progress and
simplification in comparison to earlier attempts \cite{Lyu,D00,D02}
while being more general at the same time, the number of steps were
still formidable, and another simplification was suggestive. This is
precisely what we want to describe now. Here, the golden key emerges
from also considering the time evolution of $R^{\pa}_{\cB}
(\omega^{\pa}_{t})$ for an arbitrary $\cB\in\PP (S)$, where
$\omega^{\pa}_{t}$ is a solution of the recombination equation
\eqref{eq:reco-eq}. In view of Propositions~\ref{prop:gen-props} and
\ref{prop:gensol}, it suffices to look at probability measures, so
that we get
\[
    \frac{\dd}{\dd t}\ts R^{\pa}_{\cB} (\omega^{\pa}_{t}) \, = \,
    \frac{\dd}{\dd t} \left(  \omega^{B_{1}}_{t}
    \otimes \cdots \otimes  \omega^{B_{\lvert \cB \rvert}}_{t} \right) 
    \, = \,  \sum_{i=1}^{\lvert \cB \rvert}\ts
    \Bigl((\pi^{\pa}_{\nts B_{i}} \nts . \ts \dot{\omega}^{\pa}_{t}) \otimes
    \bigotimes_{j\ne i} \omega^{B_{j}}_{t} \Bigr) .
\]
To proceed, it will be instrumental to understand the behaviour of
recombination on subsystems defined by a set $\varnothing \ne U \nts
\subseteq S$, where we begin by recalling \cite[Lemma~1]{BBS}.  Note
that this result effectively underlies assertion (7) of
Proposition~\ref{prop:gen-props}.

\begin{lemma}\label{lem:technical}
  Let\/ $S$ be a finite set as above, and\/ $\cA = \{ A_{1}, \dots ,
  A_{\lvert \cA \rvert} \} \in \PP(S)$ an arbitrary partition. If\/
  $U\nts\subseteq S$ is non-empty and $\omega\in\cM_{+} (X)$, one has
\[
     \pi^{\pa}_{U} . \bigl( R^{\ts S}_{\nts \cA}  (\omega) \bigr) \, = \, 
     R^{\ts U}_{\nts \cA|^{\pa}_{U}} (\pi^{\pa}_{U}  . \ts \omega) \ts ,
\]   
where\/ $\cA|^{\pa}_{U} \in \PP(U)$ and the upper index of a
recombinator indicates on which measure space it acts, with\/
$R^{S}_{\nts \cA} = R^{\pa}_{\nts \cA}$.  \qed
\end{lemma}

To continue, it is clear that we will need the recombination rates on
subsystems defined by some $\varnothing \ne U \subseteq S$, as induced
by the marginalisation
\begin{equation}\label{eq:marg-rates}
    \varrho^{U} \! (\cA) \, = 
     \sum_{\substack{\cB \in \PP(S) \\ \cB|^{\pa}_{U} = \ts \cA}} 
     \! \varrho^{S} (\cB) \ts ,
\end{equation}
where $\varrho^{S} (\cB) = \varrho (\cB)$.  Now,
Lemma~\ref{lem:technical} implies the following marginalisation
consistency on the level of probability measures, where we use the
notation $\omega^{U}_{t} = \pi^{\pa}_{U} \nts . \ts \omega^{\pa}_{t}$
as introduced earlier. The version we state here is a special case of
\cite[Prop.~6]{BBS}; see also Lemma~\ref{lem:marg-discrete} below.

\begin{prop}\label{prop:marg}
  Let\/ $\varnothing \ne U \subseteq S$.  If\/ $\omega^{\pa}_{t}$ is a
  solution of the Cauchy problem of Eq.~\eqref{eq:reco-eq} with
  initial condition\/ $\omega^{\pa}_{0} \in \cP (X)$, the
  marginal measures\/ $(\omega^{U}_{t})^{\pa}_{t\geqslant 0}$ on\/
  $X^{\pa}_{U}$ solve the ODE
\[
      \frac{\dd}{\dd t} \, \omega^{U}_{t} \,= \! \sum_{\cA \in \PP (U)}\!
      \varrho^{U} \! (\cA) \, \bigl( R^{\ts U}_{\nts \cA} - \one \bigr)
      (\omega^{U}_{t})
\]
with initial condition\/ $\omega^{U}_{0} = \pi^{\pa}_{U}\nts . \ts
\omega^{\pa}_{0}$ and marginalised rates\/ $\varrho^{U} \! ( \cA)$
according to Eq.~\eqref{eq:marg-rates}. In particular, one has\/
$\omega^{U}_{t}\! \in \cP (X^{\pa}_{U})$ for all\/ $t\geqslant 0$.  \qed
\end{prop}       

In this context, it is helpful to also note a factorisation property
of the recombinators on $\cM_{+} (X)$.

\begin{lemma}\label{lem:factor}
  Let\/ $\{U , V \}$ be a partition of $S$, and assume that two
  partitions $\cA\in\PP(U)$ and $\cB\in\PP(V)$ are given. Then, for
  any\/ $0 \ne \mu \in \cM_{+} (X)$, one has
\[
     R^{\pa}_{\nts\cA\sqcup\cB} \ts (\mu) \; = \;
     \frac{1}{\| \mu \|} \,
     R^{\ts U}_{\nts \cA} (\mu^{U}) \otimes 
     R^{V}_{\cB} (\mu^{\nts V}) \ts , 
\]
which simplifies to $R^{\pa}_{\nts\cA\sqcup\cB} \ts (\mu) = 
R^{\ts U}_{\nts \cA} (\mu^{U}) \otimes R^{V}_{\cB} (\mu^{\nts V})$
for\/ $\mu \in \cP (X)$.
\end{lemma}

\begin{proof}
  Observe first that $\{ U,V\} \wedge (\cA \sqcup \cB) = \cA \sqcup
  \cB$ due to our assumptions. By assertions (1) and (7) from
  Proposition~\ref{prop:gen-props}, we then know that
\[
    R^{\pa}_{\nts\cA\sqcup\cB} \ts (\mu) \, = \,  \bigl(
    R^{\pa}_{\nts\cA\sqcup\cB} \, R^{\pa}_{\{ U,V\}} \bigr)
    (\mu) \, = \, R^{\pa}_{\nts\cA\sqcup\cB} \bigl(
    \tfrac{1}{\| \mu \|} \, \mu^{U} \! \otimes \mu^{V} \bigr)
    \, = \,  \frac{1}{\| \mu \|}     
    R^{\ts U}_{\nts \cA} (\mu^{U}) \otimes R^{V}_{\cB} 
     (\mu^{\nts V}) \ts , 
\] 
from which the second claim is immediate.
\end{proof}

We are now set to proceed with solving Eq.~\eqref{eq:reco-eq}.

\section{Solution of the recombination 
equation in continuous time}\label{sec:fast} 

Let us consider the time evolution of $R^{\pa}_{\cB}
(\omega^{\pa}_{t})$ for an arbitrary partition $\cB \in \PP (S)$, here
written as $\cB = \{B^{\pa}_{1}, \ldots, B^{\pa}_{m}\}$. As before, we
assume $\omega^{\pa}_{t}$ to be a solution of the recombination
equation \eqref{eq:reco-eq} with initial condition $\omega^{\pa}_{0}
\in \cP (X)$.  Using the product rule as above, and employing
Proposition~\ref{prop:marg}, we obtain
\begin{eqnarray} 
     \frac{\dd}{\dd t} \ts R^{\pa}_{\cB}  ( \omega^{\pa}_{t}  ) 
    & = & \sum_{i=1}^{m}  \Bigl( \frac{\dd}{\dd t} \, 
       \omega^{B_{i}}_{t} \Bigr)  \otimes
        \bigotimes_{j\ne i} \omega^{B_{j}}_{t} \nonumber \\[1mm]
    & = & \sum_{i=1}^{m} \, \sum_{\cA_{i} \in \PP (B_{i})} \!
        \varrho^{B_{i}} (\cA_{i}) \, \bigl( R^{B_{i}}_{\! \cA_{i}} 
               - \one \bigr) (\omega^{B_{i}}_{t}) \otimes
         \bigotimes_{j\ne i} \omega^{B_{j}}_{t} \nonumber \\[1mm]
    & = & \sum_{i=1}^{m} \, \sum_{\cA_{i} \in \PP(B_i)} \!
       \varrho^{B_{i}} (\cA_{i})  
       \bigl( R^{\pa}_{(\cB \setminus B_i) \sqcup\ts \cA_i} 
        -  R^{\pa}_{\cB} \bigr)  (\omega^{\pa}_{t}) \nonumber \\[1mm]
    & = & \sum_{i=1}^{m} \,
       \sum_{\substack{\cA_{i} \in \PP(B_i) \\ \cA_{i} \ne \{ B_{i}\} } }
        \! \varrho^{B_{i}} (\cA_{i})  
       \bigl( R^{\pa}_{(\cB \setminus B_i) \sqcup\ts \cA_i} 
        -  R^{\pa}_{\cB} \bigr)  (\omega^{\pa}_{t}) \ts ,
        \label{eq:reco-fast}   
 \end{eqnarray}
 where $\cB\setminus B_{i}$ denotes the partition of $S \setminus
 B_{i}$ that emerges from $\cB$ by removing $B_{i}$.  Note that the
 crucial third step follows from Lemma~\ref{lem:factor} used
 backwards. Clearly, we may restrict the inner summation to $\cA_{i}
 \ne \{ B_{i}\}$, as the then omitted term vanishes anyhow, which
 gives the last line. We can now state our main result as follows.

\begin{theorem}\label{thm:main}
  Let\/ $\omega^{\pa}_{t}$ be a solution of the recombination equation
  \eqref{eq:reco-eq}, with initial condition\/ $\omega^{\pa}_{0} \in
  \cP (X)$.  Then, for any partition\/ $\cB \in \PP(S)$, the measure\/
  $R^{\pa}_{\cB} (\omega^{\pa}_{t})$ satisfies an ODE of the linear 
  form
\[
    \frac{\dd}{\dd t}\ts R^{\pa}_{\cB}  ( \omega^{\pa}_{t} )  
    \, = \sum_{\cC \in \PP (S)} Q^{\pa}_{\cB \cC} 
    \ts R^{\pa}_{\cC} (\omega^{\pa}_{t}) \ts ,
\]  
   where the coefficients are explicitly given by
\begin{equation}\label{eq:Q}
  Q^{\pa}_{\cB \cC} :=  \begin{cases}
     \varrho^{B_i} (\cA^{}_i), & 
        \text{if } \ts \cC = (\cB \setminus  B^{}_i) \sqcup \cA_{i} 
        \text{ for some } \{ B_{i} \} \ne \cA_i \in \PP(B_i) \\ &
        \text{and precisely one index } 1 \leqslant i \leqslant 
        \lvert \cB \rvert \ts , \\
    - \sum\limits_{i=1}^{\lvert \cB \rvert}
        \sum\limits_{\substack{\cA_{i} \ne \{ B_{i} \} \\ \cA_{i} \in \PP(B_{i})}  }
        \! \! \varrho^{B_{i}} (\cA_{i}), & \text{if } \cC = \cB \ts , \\
         0, & \text{otherwise}\ts .     \end{cases}
\end{equation}
In particular, the matrix\/ $Q = \bigl( Q^{\pa}_{\cB \cC}
\bigr)_{\cB,\cC\in\PP(S)}$ is a Markov generator with triangular
structure, and\/ $R^{\pa}_{\cB} (\omega^{\pa}_{t})$ is a probability
measure for all\/ $t \geqslant 0$.
\end{theorem}  
  
\begin{proof}
  It is clear from Eq.~\eqref{eq:reco-fast} that the derivative of
  $R^{\pa}_{\cB} (\omega^{\pa}_{t})$ can indeed be written as a 
  linear combination, namely
\[
    \frac{\dd}{\dd t}\ts R^{\pa}_{\cB}  ( \omega^{\pa}_{t}) \, 
       = \sum_{\cC \prec \cB} Q^{\pa}_{\cB \cC} 
          \bigl( R^{\pa}_{\cC}  -  R^{\pa}_{\cB} \bigr) 
             (\omega^{\pa}_{t}) \, 
       =  \sum_{\cC \preccurlyeq \cB} Q^{\pa}_{\cB \cC} \ts
          R^{\pa}_{\cC} (\omega^{\pa}_{t}) \,
        =  \! \sum_{\cC \in \PP (S)} \! \! Q^{\pa}_{\cB \cC} \ts
          R^{\pa}_{\cC} (\omega^{\pa}_{t})   \ts ,
\]
where the second step follows from a simple change of summation, while
the third just reflects the fact that $Q$ has a triangular
structure. It remains to show that the coefficients $Q^{\pa}_{\cB
  \cC}$ are those given by Eq.~\eqref{eq:Q}.

If $\cC \prec \cB$, Eq.~\eqref{eq:reco-fast} tells us that this
coefficient must vanish unless $\cC$ refines precisely one part of
$\cB$, in which case its value is as claimed, and non-negative due to
our general assumption on the recombination coefficients.  When
$\cC=\cB$, we read from our change of summation that all non-diagonal
coefficients of the row of $Q$ defined by $\cB$ must occur on the
diagonal once, with negative sign. All other coefficients clearly
vanish.

The Markov generator property is then clear, and the last claim
follows from our general properties of the recombinators in
conjunction with Proposition~\ref{prop:gensol}.
\end{proof}  
  
The meaning of the Markov generator $Q$ will become clear in the next
section. Let us now define the (column) vector $\varphi^{\pa}_t :=
\bigl( \varphi^{\pa}_{t} (\cB) \bigr)_{\cB \in \PP(S)}$ with $
\varphi^{\pa}_{t} (\cB) := R^{\pa}_{\cB} (\omega^{\pa}_{t})$. With
this abbreviation, the ODEs from Theorem~\ref{thm:main} now turn into
the \emph{linear} ODE system
\[
    \frac{\dd}{\dd t} \,
    \varphi^{\pa}_{t} \, = \, Q \ts \varphi^{\pa}_{t} 
\]
with initial condition $\varphi^{\pa}_{0}= \bigl( R^{\pa}_{\cB}
(\omega^{\pa}_{0}) \bigr)_{\cB \in \PP(S)}$ and solution
\[
     \varphi^{\pa}_{t} \, = \, \ee^{t \ts Q} \ts \varphi^{\pa}_{0} \ts .
\]
Note that $\{ \ee^{t\ts Q} \mid t \geqslant 0 \}$ is the Markov
semigroup generated by $Q$. In particular, for the first component of
$\varphi^{\pa}_{t}$, we now get
\begin{equation}\label{eq:reco-sol}
     \omega^{\pa}_{t} \, = \, \varphi^{}_t (\pmax) \, =  
     \! \sum_{\cA\in\ts\PP(S)}\! \! a^{\pa}_{t} 
     (\cA) \, R^{\pa}_{\cA} (\omega^{\pa}_{0}),
\end{equation}
with $a^{\pa}_{t} (\cA) = (\ee^{t \ts Q})^{\pa}_{\pmax \cA}$, which
leads us back to Eq.~\eqref{eq:ansatz}.  Clearly, $\sum_{\cA \in
  \PP(S)} a^{\pa}_{t} (\cA) =1$ since each $\ee^{t \ts Q}$ with $t
\geqslant 0$ is a Markov matrix, so all row sums are $1$.

Via the Kolmogorov forward equation for the Markov semigroup $\{ e^{t
  Q} \mid t \geqslant 0 \}$, compare \cite[Thm.~2.1.1]{Norris}, the
following consequence is now immediate; see \cite{BBS,interval} for
the original (but much longer) derivation.
\begin{coro}
   Under the assumptions of Theorem~$\ref{thm:main}$, the coefficient
   functions of Eq.~\eqref{eq:reco-sol} are\/ $a^{\pa}_{t} (\cA) = 
   ( e^{t Q})^{\pa}_{\pmax \cA}$, with the Markov 
   generator\/ $Q$ as in Eq.~\eqref{eq:Q}, and satisfy the ODEs
\[
     \dot{a}^{\pa}_{t} (\cA) \, = \sum_{\cB \succcurlyeq \cA}
     a^{\pa}_{t} (\cB) \, Q^{\pa}_{\cB \cA}
\]   
   with initial conditions\/ $a^{\pa}_{0} (\cA) = 
   \delta^{\pa}_{\pmax \cA}$.   In particular, $\bigl(a^{\pa}_{t} 
   (\cA)\bigr)_{\cA\in \PP(S)}$  is a probability vector for
   all\/ $t\geqslant 0$.    \qed
\end{coro}

Let us pause to discuss what Eq.~\eqref{eq:reco-sol} tells us. First,
the solution of the recombination equation may be expressed in terms
of a \emph{convex combination} of the initial measure recombined in
all possible ways. This is quite plausible, given that the
differential equation means a continuous replacement of the current
measure by its recombined versions. Second, the procedure just
described has uncovered a \emph{linear} structure that underlies the
nonlinear recombination equation, and thus reduced the problem to a
linear one. With hindsight, we recognise a streamlined version of
Haldane linearisation: The $R^{\pa}_{\cB} (\omega^{\pa}_{t})$ with
$\cB\ne \pmax$ can be viewed as additional components that are used to
enlarge the system in order to unravel its intrinsic linear
structure. The latter is conveyed by the Markov generator $Q$, whose
meaning still remains to be elucidated, as will be done next.

\section{The backward point of view: Partitioning 
process}\label{sec:partitioning}

Now that we have understood the structure of the ODE system for the
coefficients $a^{}_t$ in the usual (forward) direction of time, let 
us consider a related (stochastic) process that will provide an
additional meaning for $a^{}_t$.  Let
$\{\varSigma^{}_t\}^{\pa}_{t\geqslant 0}$ be a Markov chain in
continuous time with values in $\PP(S)$ that is constructed as
follows. Start with $\varSigma^{}_0= \pmax$. If the current state is
$\varSigma^{}_t=\cB$, then part $B_{i}$ of $\cB$, with $1 \leqslant i
\leqslant |\cB|$, is replaced by $\{ B_{i} \} \ne \cA_i \in \PP(B_i)$
at rate $\varrho^{B_i} (\cA_i)$, independently of all other
parts. That is, the transition from $\cB$ to $(\cB \setminus B_{i})
\sqcup \cA_{i}$ occurs at rate $\varrho^{B_{i}} (\cA_{i})$ for all $\{
B_{i} \} \ne \cA_{i} \in \PP(B_{i}) $ and $1 \leqslant i \leqslant
|\cB|$.  Put differently, the transition from $\cB$ to $\cC$ happens
at rate $Q^{\pa}_{\cB\cC}$ of Eq.~\eqref{eq:Q}.

This way, we have given a meaning to the generator $Q$ of
Section~\ref{sec:fast}: It holds the transition rates of the process
of progressive refinements $\{\varSigma_{t} \}^{\pa}_{t \geqslant 0}$,
which we have just described, and which we call the underlying
\emph{partitioning process}. The argument is illustrated in
Figure~\ref{fig:tree}.

Since $Q$ is the Markov generator of
$\{\varSigma_t\}^{\pa}_{t\geqslant 0}$, we can further conclude that
\[
    (\ee^{t \ts Q})^{\pa}_{\cB \cC} \, = 
    \, \mathbf{P} \bigl( \varSigma_{t} = 
    \cC \mid \varSigma_0 = \cB \bigr)
\]    
(where $\mathbf{P}$ denotes probability), that is, the transition
probability from `state' $\cB$ to `state' $\cC$ during a time interval
of length $t$.  In particular, $a^{\pa}_{t} (\cA) = \bigl(\ee^{t \ts
  Q} \bigr)_{\pmax\cA} = \mathbf{P} \bigl( \varSigma^{\pa}_{t} = \cA
\mid \varSigma^{\pa}_{0} = \pmax \bigr)$.  We have therefore shown our
second main result, which can now be stated as follows.

\begin{theorem}\label{thm:a-part}
  The probability vector\/ $a^{\pa}_{t}$ from Eqs.~\eqref{eq:ansatz}
  and \eqref{eq:reco-sol} agrees with the distribution of the
  partitioning process\/ $\{\varSigma_{t} \}^{\pa}_{t \geqslant 0}$.
  Explicitly, we have
\[
      a^{\pa}_{t} (\cA) \, = \, 
     \mathbf{P} \bigl(\varSigma^{\pa}_{t} = \cA \mid 
        \varSigma^{\pa}_{0} = \pmax \bigr)
\]
for any\/ $\cA \in \PP(S)$ and all\/ $t \geqslant 0$.  \qed
\end{theorem}

\begin{figure}
\includegraphics[width=\textwidth]{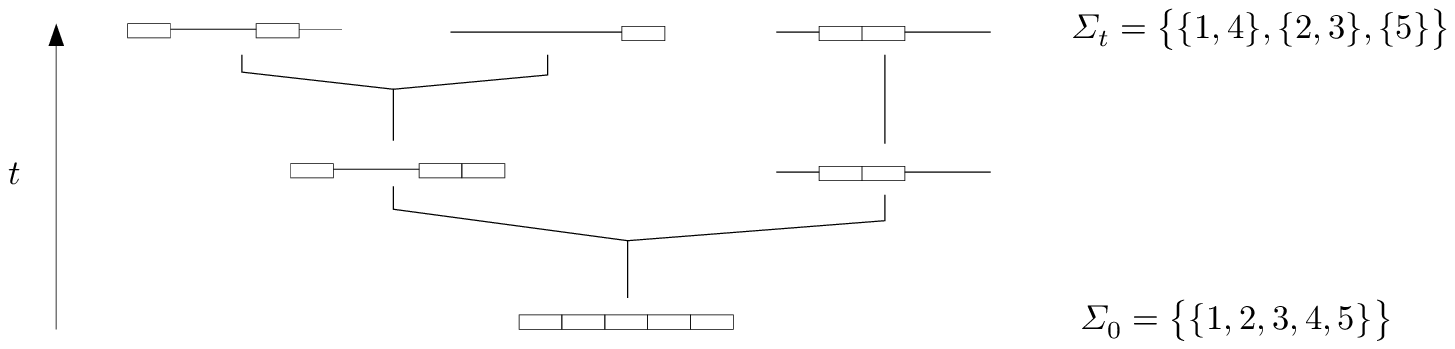}
\caption{\label{fig:tree} Illustrative realisation of the partitioning
  process.}
\end{figure}

\section{Recombination equation in discrete 
time}\label{sec:discrete}

Let us turn our attention to the discrete-time analogue of
Eq.~\eqref{eq:reco-eq}, which is often considered in population
genetics \cite{Buerger, D02,Lyu,SM}. We can use the same general setting 
of Section~\ref{sec:partitions}, in particular the space $\cM (X)$ of
measures and the action of the recombinators on it. Then, one has to
consider the nonlinear iteration
\begin{equation}\label{eq:reco-discrete}
   \omega^{\pa}_{t+1} \, = \!  \sum_{\cA \in \PP (S)} \!
   r (\cA) \, R^{\pa}_{\cA} (\omega^{\pa}_{t}) \ts ,
\end{equation}
where the parameters are now recombination \emph{probabilities}, so $r
(\cA) \geqslant 0$ for $\cA\in \PP (S)$ together with
$\sum_{\cA\in\PP(S)} r (\cA) = 1$. Moreover, $t\in \NN_{0}$ denotes
discrete time (counting generations, say) with initial condition
$\omega^{\pa}_{0}$. Clearly, the positive cone $\cM_{+} (X)$ is
preserved under the iteration, as is the norm of a positive
measure. In view of the general properties of the recombinators from
Proposition~\ref{prop:gen-props}, we can further confine our
discussion to $\omega^{\pa}_{0}\in \cP (X)$, which immediately implies
that $\omega^{\pa}_{t} \in \cP (X)$ for all $t\in\NN$ as well.

As above in Proposition~\ref{prop:marg}, we need marginalisation
consistency for subsystems. A generalisation of \cite[Eq.~2.5]{Chris}
to our setting leads to the following result.
\begin{lemma}\label{lem:marg-discrete}
  Let\/ $\omega^{\pa}_{t}$ with\/ $t\in\NN_{0}$ be a solution of the
  discrete recombination equation \eqref{eq:reco-discrete}, with
  initial condition\/ $\omega^{\pa}_{0}\in \cP (X)$. Then, for any\/
  $\varnothing \ne \nts U \subseteq S$, the marginal measures\/
  $\omega^{U}_{t} \! $ satisfy the induced recombination equation
\[
     \omega^{\ts U}_{t+1} \, = \! \sum_{\cA\in\PP (U)} \!\!
     r^{U} \! (\cA) \, R^{\ts U}_{\nts \cA} (\omega^{U}_{t}) \ts ,
\]    
    for all\/ $t\in\NN_{0}$, with\/ $R^{\ts U}_{\nts \cA}$ as before.
    Here,
\[
      r^{U}\! (\cA) \; := \! \sum_{\substack{\cB\in\PP(S) \\ 
        \cB|^{\pa}_{U} = \cA}} \! \! r (\cB)  \, \geqslant \, 0
\]   
    are the induced recombination probabilities for the subsystem,
    with\/ $\sum_{\cA\in\PP (U)} r^{U}\! (\cA) = 1$.
\end{lemma}

\begin{proof}
The claim follows from a simple calculation,
\[
\begin{split}
    \omega^{U}_{t+1} \; & = \; \pi^{\pa}_{U}  . \ts \omega^{\pa}_{t+1}
    \; = \; \pi^{\pa}_{U} \ts .  \! \sum_{\cB\in\PP (S)} \! r (\cB) \,
    R^{\pa}_{\cB} (\omega^{\pa}_{t}) \; = \sum_{\cB\in\PP (S)} \! r(\cB) \,
    \bigl(\pi^{\pa}_{U} . \ts R^{\pa}_{\cB} (\omega^{\pa}_{t}) \bigr) \\[1mm]
    & = \sum_{\cB\in\PP (S)} \! r(\cB) \, R^{\ts U}_{\cB|^{}_{U}}\nts
    (\omega^{U}_{t}) \; = \! \sum_{\cA\in\PP (U)} \, 
    \sum_{\substack{\cB \in\PP (S) \\ \cB|^{}_{U} = \cA}}
    \! r(\cB) \, R^{\ts U}_{\nts \cA} (\omega^{U}_{t}) \; = \!
    \sum_{\cA\in\PP (U)} \! r^{U} \! (\cA) \, R^{\ts U}_{\nts \cA}
    (\omega^{U}_{t}) \ts ,
\end{split}
\]
where the first step in the second line is a consequence of 
Lemma~\ref{lem:technical}.
\end{proof}

Let now $\cB = \{ B^{\pa}_{1}, \ldots , B^{\pa}_{m}\}$ be an arbitrary
partition of $S$, and consider
\[
\begin{split}
    R^{\pa}_{\cB} (\omega^{\pa}_{t+1}) \, & = \,
    \omega^{B_{1}}_{t+1} \otimes \dots \otimes
    \omega^{B_{m}}_{t+1} \, = \; \bigotimes_{i=1}^{m} \,
     \sum_{\cA_{i}\in\PP(B_{i})} \!\! r^{B_{i}} (\cA_{i}) \, 
    R^{B_{i}}_{\nts \cA_{i}} (\omega^{B_{i}}_{t})  \\[1mm]
    & = \sum_{\cA_{1} \in \PP (B_{1})}  \! \cdots \!\!
        \sum_{\cA_{m}\in\PP(B_{m})} \biggl(\,
         \prod_{i=1}^{m}  r^{B_{i}} (\nts \cA_{i}) \biggr)
        \, R^{B_{1}}_{\nts \cA_{1}} (\omega^{B_{1}}_{t}) \otimes
        \dots \otimes R^{B_{m}}_{\nts \cA_{m}} (\omega^{B_{m}}_{t}) \ts ,
\end{split}
\]
where we have used Lemma~\ref{lem:marg-discrete}. Invoking the
factorisation property from Lemma~\ref{lem:factor}, we thus see that
we can rewrite the last expression as a linear combination of terms of
the form $R^{\pa}_{\cC} (\omega^{\pa}_{t})$ with $\cC \in \PP(S)$
being a refinement of $\cA_{1} \sqcup \dots \sqcup \cA_{m}$. We can
now formulate the following general result, which also resembles some
recent findings from \cite{SM}.

\begin{theorem}\label{thm:discrete}
  Let\/ $\omega^{\pa}_{t}$ with\/ $t\in\NN_{0}$ be a solution of the
  discrete recombination equation \eqref{eq:reco-discrete}, with
  initial condition\/ $\omega^{\pa}_{0}\in \cP (X)$. Then, for any\/
  $\cB\in\PP (S)$ and any\/ $t\in\NN_{0}$, the measure\/
  $R^{\pa}_{\cB} (\omega^{\pa}_{t})$ is a probability measure and
  satisfies the linear recursion
\[
    R^{\pa}_{\cB} (\omega^{\pa}_{t+1}) \;  =  \,
    \sum_{\cC\preccurlyeq \cB}  M^{\pa}_{\cB \cC} \,
    R^{\pa}_{\cC} (\omega^{\pa}_{t}) \; = 
    \sum_{\cC \in \PP (S)} \! M^{\pa}_{\cB \cC} \,
    R^{\pa}_{\cC} (\omega^{\pa}_{t}) 
\]
with the coefficients
\begin{equation}\label{eq:M-mat}
    M^{\pa}_{\cB \cC} \, = \, \begin{cases}
    \prod_{i=1}^{\lvert \cB \rvert} r^{B_{i}} (\cC |^{\pa}_{B_{i}}) ,
    & \text{if } \cC \preccurlyeq \cB , \\
    0 , & \text{otherwise}. \end{cases}
\end{equation}
In particular, $M = \bigl(M^{\pa}_{\cB \cC}\bigr)_{\cB,\cC\in\PP (S)}$
is a triangular Markov matrix.
\end{theorem}  

\begin{proof}
  The first claim, as mentioned earlier, is a direct consequence of
  part (6) of Proposition~\ref{prop:gen-props}.  Our above
  calculation, in conjunction with Lemma~\ref{lem:factor}, proves the
  second claim, where the determination of the coefficients
  $M^{\pa}_{\cB \cC}$ is a straight-forward exercise.

  Clearly, we have $M^{\pa}_{\cB \cC} \geqslant 0$ for all $\cB,\cC
  \in \PP (S)$ because $r (\cA) \geqslant 0$ by assumption, hence also
  $r^{U} \! (\cB) \geqslant 0$ for all $\varnothing \ne \nts U
  \subseteq S$ by the formula in Lemma~\ref{lem:marg-discrete}. It
  remains to show that each row of $M$ sums to $1$. Indeed, given any
  $\cB \in \PP (S)$, one has
\[
   \sum_{\cC\in\PP (S)} \! M^{\pa}_{\cB \cC} \, = \sum_{\cC\preccurlyeq \cB}
   M^{\pa}_{\cB \cC} \, = \, \prod_{i=1}^{\lvert \cB \rvert} \,
   \sum_{\cA_{i}\in \PP (B_{i})}  \!\! r^{B_{i}} (\cA_{i}) \, = \, 1 
\]
because $\bigl( r^{U} \! (\cA) \bigr)_{\cA \in \PP (U)}$ is a
probability vector for all $\varnothing\ne\nts U \subseteq S$ by
Lemma~\ref{lem:marg-discrete}.
\end{proof}

Note that the matrix entry $M^{\pa}_{\cB \cB}$ is the probability that
`nothing happens' to partition $\cB$ in one step. These diagonal
entries are, due to the triangular structure, the eigenvalues of $M$.
For a special case of our setting, they have been determined earlier,
by rather different methods and without reference to a linear
structure, in \cite[Thm.~6.4.3]{Lyu}; see also \cite{NHB,WBB}.
Theorem~\ref{thm:discrete} now gives them a clearer meaning in a more
general setting. Also, the degenerate cases that had to be excluded in
previous attempts \cite{Lyu,D02} precisely correspond to the cases
where $M$ fails to be diagonalisable. They pose no problem in the
above approach.

At this point, we can repeat our previous interpretation. If
$\varphi^{\pa}_{t} = \bigl( R^{\pa}_{\nts \cA} (\omega^{\pa}_{t})
\bigr)_{\cA\in\PP (S)}$ is considered as a column vector, with
$\omega^{\pa}_{t}$ a solution of the recombination equation, we get
$\varphi^{\pa}_{t+1} = M \varphi^{\pa}_{t}$, and hence
\[
      \varphi^{\pa}_{t} \, = \, M^{t} \ts \varphi^{\pa}_{0}
\]
for all $t \in \NN_{0}$. The first component of the vector is
\[
    \omega^{\pa}_{t} \, = \, \varphi^{\pa}_{t} (\pmax)
    \, = \, \bigl( M^{t} \varphi^{\pa}_{0}\bigr) (\pmax)
    \, = \! \sum_{\cA\in\PP (S)} \! a^{\pa}_{t} (\cA) \,
    R^{\pa}_{\cA} (\omega^{\pa}_{0}) \ts ,
\]
with $a^{\pa}_{t} (\cA) = (M^{t})^{\pa}_{\pmax \cA}$. In particular,
$a^{\pa}_{0} (\cA) = \delta^{\pa}_{\pmax \cA}$.
\smallskip

There is again an underlying stochastic process, in analogy to the
continuous-time case in Section~\ref{sec:partitioning}. Here, $\{
\varSigma^{\pa}_{t} \}^{\pa}_{t \in \NN_{0}}$ is a Markov chain in
discrete time with values in $\PP (S)$, starting at
$\varSigma^{\pa}_{0} = \pmax$. When $\varSigma^{\pa}_{t} = \cB$, in
the time step from $t$ to $t+1$, part $B_{i}$ of $\cB$ is replaced by
$\cA_{i} \in \PP (B_{i})$ with probability $r^{B_{i}}_{\cA_{i}}$,
independently for each $1 \leqslant i \leqslant \lvert \cB
\rvert$. Note that, unlike in the continuous-time case, \emph{several}
parts can be refined in one step, which makes the discrete-time case
actually more complicated. Of course, $\cA_{i} = \{ B_{i}\}$, which
means no action on this part, is also possible. Put together, it is
not difficult to verify that one ends up precisely with the transition
matrix $M$ from Eq.~\eqref{eq:M-mat}.

\section*{Acknowledgements}
This work was supported by the German Research Foundation (DFG),
within the SPP 1590.

\end{document}